\documentclass[12pt]{amsart}
\usepackage{amsmath,amsthm,amsfonts,latexsym,amssymb,mathrsfs,fancyhdr,graphicx}
\usepackage{framed}
\input xy
\xyoption{all}

\theoremstyle{theorem}
\newtheorem{theorem}{Theorem}

\newtheorem{lemma}{Lemma}
\newtheorem{proposition}{Proposition}

\long\def\symbolfootnote[#1]#2{\begingroup%
\def\thefootnote{\fnsymbol{footnote}}\footnote[#1]{#2}\endgroup}

\theoremstyle{definition}
\newtheorem{definition}{Definition}
\newtheorem{remark}{Remark}

\numberwithin{equation}{section}

\begin{document}

\title[Unifying some classical results on Artinian rings and modules]
{Unifying some classical results on Artinian rings and modules}

\author{Donovan Leyba}
\address[Donovan Leyba]{Department of Mathematics\\
University of Colorado\\
Colorado Springs, CO 80918\\
USA}
\email{dleyba@uccs.edu}

\author{Zachary Mesyan}
\address[Zachary Mesyan]{Department of Mathematics\\
University of Colorado\\
Colorado Springs, CO 80918\\
USA}
\email{zmesyan@uccs.edu}

\author{Greg Oman}
\address[Greg Oman]{Department of Mathematics\\
University of Colorado\\
Colorado Springs, CO 80918\\
USA}
\email{goman@uccs.edu}

\symbolfootnote[0]{\emph{2020 Mathematics Subject Classification.} Primary: 13E10; Secondary: 13E05, 13A15, 16P20, 16P40.}
\symbolfootnote[0]{\emph{Key Words and Phrases.} Artinian ring, Artinian module, localization, minimal prime ideal, Noetherian ring, Noetherian module}

\begin{abstract} In this note, we introduce a very crude but natural notion of measure on the class of left $R$-modules over a ring $R$. We use this notion to give short proofs of some classical theorems on (left) Artinian rings and modules, due to Akizuki, Anderson, Hopkins, and Levitzki, as well as of some new results. \end{abstract}

\maketitle

\section{Introduction}

The notion of size is, of course, ubiquitous in mathematics. In fact, this concept arises in virtually every mathematical area. In set theory, there are so-called ``large" cardinals, which yield a way of counting the number of objects in sets of massive size. Analytically, one has the notion of divergence of series of positive terms, whose partial sums grow without finite bound. On the other end of the spectrum, one frequently encounters notions of smallness, often relative to some superset. For example, subsets of the real line of measure zero are, in some sense, ``sparse". Number-theoretically, the set of prime numbers is well-known to have natural density zero (relative to the set of positive integers). There are many other examples we could mention, which would take us too far afield.

In this paper, we introduce a simple but versatile measure of size for left modules over a ring--see Definition~\ref{measure-def} for more details. Using this idea we give unified short proofs of various classical results. Specifically, every Artinian left module over a left Artinian ring is Noetherian (Hopkins-Levitzki), every commutative zero-dimensional Noetherian ring is Artinian (Akizuki), and every Artinian module over a commutative ring is countably generated (Anderson). We also use the ideas introduced in the paper to give a description of the possible cardinality of an Artinian module over a commutative ring.

\section{Preliminaries}

We begin by remarking that \textbf{throughout this paper, all rings are assumed to be with identity, and all modules are assumed to be unitary.}

There are various properties of a left module $M$ over a ring $R$ which one may regard as representing a kind of ``largeness" or ``smallness." For example, one may informally view left modules possessing either the ascending or descending chain condition on submodules as being ``small", and left modules which are not countably generated as being ``large." We pause to collect the following facts (some of which are simply reformulations of standard ones found in most any graduate algebra text).

First, we remind the reader that if $M$ is a left module over a ring $R$ and $\kappa$ is an infinite cardinal, then $M$ is \emph{$\kappa$-generated} provided there is a generating set for $M$ of cardinality $\kappa$ but none of cardinality strictly less than $\kappa$. We denote the cardinality of a set $X$ by $|X|$.

\begin{lemma}\label{lem1} Let $R$ be a ring, let $M$ be a left $R$-module, let $N$ be a submodule of $M$, and let $\kappa$ be an infinite cardinal.

\begin{enumerate}
\item If $M$ is not Noetherian but $N$ is Noetherian, then $M/N$ is not Noetherian.
\item If $M$ is not Artinian but $N$ is Artinian, then $M/N$ is not Artinian.
\item If $|M|=\kappa$ and $|N|<\kappa$, then $|M/N|=\kappa$.
\item If $M$ is $\kappa$-generated but $N$ is $\lambda$-generated, for some cardinal $\lambda<\kappa$, then $M/N$ is $\kappa$-generated.
\end{enumerate}

\end{lemma}

\begin{proof} (1) Consider an exact sequence $0\to A\rightarrow B\to C\to 0$ of left $R$-modules. It is well-known \cite[Chapter X, Propositions 1.1 and 1.2]{SL} that $B$ is Noetherian if and only if both $A$ and $C$ are Noetherian. Now take $A:=N$, $B:=M$, and $C:=M/N$, with the canonical maps.  It follows that if $M$ is not Noetherian but $N$ is Noetherian, then $M/N$ is not Noetherian.

\vspace{.1 in}

(2) It is also well-known \cite[Chapter X, Proposition 7.1]{SL} that the claims above hold with ``Noetherian" replaced with ``Artinian."

\vspace{.1 in}

(3) Suppose that $|M|=\kappa$ and $|N|<\kappa$. Since $|M|=|M/N|\cdot|N|$ (recall that every coset of $N$ in $M$ has the same cardinality as $N$) and $|N|<\kappa$, we see that \[\kappa=|M|=|M/N|\cdot|N|=\max(|M/N|,|N|)=|M/N|.\]

\vspace{.1 in}

(4) Assume that $M$ is $\kappa$-generated and $N$ is $\lambda$-generated, for some cardinal $\lambda<\kappa$. It is clear that $M/N$ can be generated by a set of cardinality at most $\kappa$. Suppose, seeking a contradiction, that $M/N$ can be generated by a subset of cardinality $\alpha<\kappa$, and write $M/N=\langle N+m_i\colon i<\alpha\rangle$ and $N=\langle n_i\colon i<\lambda\rangle$. It is easy to see that $M$ is generated by $\{m_i\colon i<\alpha\}\cup\{n_i\colon i<\lambda\}$, which has cardinality $<\kappa$, since $\kappa$ is infinite, contradicting $M$ being $\kappa$-generated. \end{proof}

The previous lemma gives a canonical class of examples of the phenomenon found in many areas of mathematics, which can informally stated as: ``big"/``small" is ``big." It is this idea that drives the proofs in this paper and leads to the next definition.

For a ring $R$, we denote by $R$-$\mathrm{mod}$ the category of all left $R$-modules. We understand the elements of $R$-$\mathrm{mod}$ to be the objects, rather than the morphisms.

\begin{definition} \label{measure-def}
Let $R$ be a ring. Let us call a mapping $\mu\colon R\textrm{-mod}\to\{0,1\}$ a \textbf{module measure}, provided $\mu$ satisfies the following properties.
\begin{enumerate}
\item $\mu(\{0\})=0$.
\item For any left $R$-modules $M_1$ and $M_2$: if $M_1\cong M_2$, then $\mu(M_1)=\mu(M_2)$.
\item For any left $R$-module $M$ and submodule $N$: if $\mu(M)=1$ and $\mu(N)=0$, then $\mu(M/N)=1$.
\end{enumerate}

\noindent Let $\mu$ be a module measure on $R$-$\mathrm{mod}$. We call $\mu$ \textbf{sum-small}, provided that for any left $R$-module $M$ and countable collection $\mathcal{S}$ of submodules of $M$, for which $\mu(N)=0$ for all $N\in\mathcal{S}$, we have $\mu(\sum_{N\in\mathcal{S}}N)=0$.

Say that a left $R$-module $M$ is \textbf{minimally $\mu$-big} if $\mu(M)=1$, but $\mu(N)=0$ for every $N<M$. Given left $R$-modules $N<M$, say that $N$ is \textbf{maximally $\mu$-small} (relative to $M$) if $\mu(N)=0$, but $\mu(K)=1$ whenever $N<K\leq M$. 
\end{definition}

A module measure, in particular, splits $R$-$\mathrm{mod}$ into two strictly full subcategories (i.e., subcategories closed under isomorphisms and containing all the morphisms between their objects in $R$-$\mathrm{mod}$). It is the ``connecting" condition (3), however, that gives module measures their main power.

The following observations about minimally $\mu$-big modules will be used throughout the rest of the paper.

\begin{lemma} \label{min-big-lem}
Let $R$ be a ring, let $\mu$ be a module measure on $R$-$\mathrm{mod}$, and let $M$ be a minimally $\mu$-big left $R$-module. Then the following hold.
\begin{enumerate}
\item $M$ is indecomposable, that is, it is a nonzero left $R$-module that cannot be written as a direct sum of nonzero submodules.
\item Every nonzero left $R$-module homomorphism $M \to M$ is surjective.
\end{enumerate}
\end{lemma}

\begin{proof}
(1) First, note that as a minimally $\mu$-big left $R$-module, $M$ is, by definition, nonzero. Now, suppose that $M = N_1 \oplus N_2$ for some nonzero left $R$-modules $N_1$ and $N_2$. Then $\mu(N_1 \oplus 0)= 0 = \mu(0 \oplus N_2)$, while $\mu(M/(N_1 \oplus 0)) = 1$. But $M/(N_1 \oplus 0) \cong 0 \oplus N_2$, producing a contradiction. Hence $M$ must be indecomposable.

\vspace{.1 in}

(2) Let $\phi : M \to M$ be a left $R$-module homomorphism, and suppose that $\phi(M) \neq M$. Then $\mu(\phi(M)) = 0$, since $M$ is minimally $\mu$-big. Now, $\phi(M) \cong M/\ker(\phi)$, and so $\mu(M/\ker(\phi))=0$. This means that $\ker(\phi)$ cannot be a proper submodule of $M$, for otherwise we would have $\mu(\ker(\phi))=0$, and hence $\mu(M/\ker(\phi))=1$. Therefore $M = \ker(\phi)$, making $\phi$ the zero homomorphism. Thus any nonzero homomorphism $M \to M$ must be surjective.
\end{proof}

\section{Artinian implies Noetherian}

We now give an alternative proof, using module measures, of the following classical result.

\begin{theorem}[Hopkins-Levitzki \cite{CH,JL}]\label{thm1}
Let $R$ be a left Artinian ring. Then every Artinian left $R$-module is Noetherian.
\end{theorem}

\begin{proof}
For each left $R$-module $M$, let
\[\mu(M) := \left\{ \begin{array}{ll}
0 & \text{if } M \text{ is finitely generated,}\\
1 & \text{otherwise.}
\end{array}\right.\]
It follows from Lemma~\ref{lem1}(4) that $\mu$ is a module measure on $R$-$\mathrm{mod}$.

Suppose that $M$ is a left $R$-module that is Artinian but not Noetherian. Then, in particular, $M$ must have submodules that are not finitely generated. Since $M$ is Artinian, we can find a submodule $N$ of $M$ minimal with respect to the property that $\mu(N) = 1$, making $N$ minimally $\mu$-big.

Let $J$ denote the Jacobson radical of $R$, and suppose that $JN=N$. Since $R$ is left Artinian, $J^n = 0$ for some positive integer $n$~\cite[Theorem 4.12]{L}. So
\[N = JN = J^2N = \cdots = J^nN = 0,\]
which contradicts $\mu(N) = 1$. Hence it must be the case that  $JN \neq N$. Therefore $\mu(JN) = 0$, and so $\mu(N/JN) = 1$. Since every proper submodule of $N$ is finitely generated, the same goes for $N/JN$, and so the left $R$-module $N/JN$ is minimally $\mu$-big as well. Hence, by Lemma~\ref{min-big-lem}(1), $N/JN$ is indecomposable.

Now, since $R$ is left Artinian, $R/J$ must be semisimple~\cite[Theorem 4.14]{L}, and hence every left $R/J$-module is semisimple~\cite[Theorem 2.5]{L}. In particular, $N/JN$ is semisimple as a left $R/J$-module, and is therefore a direct sum of simple left $R/J$-modules~\cite[Theorem 2.4]{L}. But $R$ and $R/J$ have the same simple left modules~\cite[Proposition 4.8]{L}, and so $N/JN$ is a direct sum of simple left $R$-modules. Since $N/JN$ is indecomposable, it follows that $N/JN$ is itself a simple left $R$-module. So, in particular, it is finitely generated, contradicting $\mu(N/JN) = 1$.

Thus if $M$ is Artinian, then it is also Noetherian.
\end{proof}

\section{Artinian implies countably generated}

The remainder of the paper is devoted to commutative rings. In that situation we can say much more about minimally $\mu$-big modules, as the following proposition shows.

\begin{proposition}\label{prop1} Let $R$ be a commutative ring, let $\mu$ be a module measure on $R$-$\mathrm{mod}$, and suppose that $M$ is a minimally $\mu$-big $R$-module. Then the annihilator $\mathrm{Ann}_R(M)$ of $M$ in $R$ is a prime ideal. Moreover, if $\mathrm{Ann}_R(M)$ is a maximal ideal of $R$, then $M$ is simple. \end{proposition}

\begin{proof} For any $r \in R$, the map $m\mapsto rm$ gives an $R$-module homomorphism $M \to M$, and so either $rM=M$ or $rM=\{0\}$, by Lemma~\ref{min-big-lem}(2). Hence, if $r,s\in R \setminus \mathrm{Ann}_R(M)$, then $rM=M$ and $sM=M$, which implies that $rsM=r(sM)=rM=M$, and therefore $rs\notin\mathrm{Ann}_R(M)$. Since $\mathrm{Ann}_R(M)$ is a proper ideal of $R$ (given that $M$ is nontrivial, by definition), this proves that $\mathrm{Ann}_R(M)$ is prime.

Next, suppose that $\mathrm{Ann}_R(M)$ is a maximal ideal of $R$. Then $M$ is an $R/\mathrm{Ann}_R(M)$-vector space, and hence a direct sum of simple $R$-modules. But, by Lemma~\ref{min-big-lem}(1), this means that $M$ must itself be a simple $R$-module. \end{proof}

Our next goal is to give an alternative proof of a result originally due to Dan Anderson, namely, that Artinian modules are countably generated. We require the following lemma from the literature, which is perhaps not as well-known. We give a short, self-contained, proof.

\begin{lemma}[Sharp \cite{RS}]\label{lem2} Let $R$ be a commutative ring, and let $M$ be an Artinian $R$-module. For every maximal ideal $J$ of $R$, let
\[M[J]:=\{m\in M\colon J^nm=\{0\} \text{ for some } n\in \mathbb{Z}^+\}.\]
Then there are finitely many maximal ideals $J_1,\ldots,J_n$  of $R$ such that $M=M[J_1]\oplus\cdots\oplus M[J_n]$. \end{lemma}

\begin{proof}
Suppose that $N$ is a cyclic submodule of the Artinian $R$-module $M$. Then $N\cong R/I$ for some ideal $I$ of $R$. It follows that $R/I$ is an Artinian, thus Noetherian, ring (by Theorem~\ref{thm1}), and hence $N$ is both Artinian and Noetherian. Therefore, $N$ possesses a composition series~\cite[Chapter VIII, Theorem 1.11]{TH}. It is immediate from this fact and mathematical induction that
\begin{equation}\label{eq3}
\textrm{for all cyclic} \ N\leq M, \ \textrm{there are maximal ideals} \ J_1,\ldots, J_k \ \textrm{for which} \ J_1J_2\cdots J_kN=\{0\}.
\end{equation}

Now, set
\begin{equation}\label{eq4}
M':=\sum_{J\in\mathrm{Max}(R)}M[J].
\end{equation}

\noindent We claim that this sum is direct. Suppose not. Then there exist distinct maximal ideals $J$, $J_1,\ldots,J_n$ and a nonzero $m\in M[J]\cap(M[J_1]+\cdots+M[J_n])$. Hence there are positive integers $k,k_1,\ldots,k_n$ and a maximal ideal $J^*$ of $R$ such that $J^k+J_1^{k_1}\cdots J_n^{k_n}\subseteq\mathrm{Ann}_R(m)\subseteq J^*$.  Because $J^*$ is prime, we conclude that $J+J_i\subseteq J^*$ for some $i$, and this is a contradiction: since $J$ and $J_i$ are distinct maximal ideals, they are coprime. This proves the directness of the sum.

Our next assertion is that $M'$ exhausts $M$. Toward this end, it suffices, by \eqref{eq3}, to show that for every positive integer $n$ and every $m\in M$: if $J_1,\ldots, J_n$ are distinct maximal ideals of $R$ such that $J_1^{k_1}\cdots J_n^{k_n}m=\{0\}$ for some positive integers $k_1,\ldots,k_n$, then $m\in M'$. We proceed by induction, the case $n=1$ being patent. Now assume that the claim is true for all $k\leq n$, and suppose that $J_1^{k_1}\cdots J_{n+1}^{k_{n+1}}m=\{0\}$, where the $J_i$ are distinct and the $k_i$ are positive integers.  Then as above, $J_1^{k_1}$ and $J_2^{k_2}\cdots J_{n+1}^{k_{n+1}}$ are relatively prime; this means that there are $x\in J_1^{k_1}$ and $y\in J_2^{k_2}\cdots J_{n+1}^{k_{n+1}}$ such that $x+y=1$. But then $xm+ym=m$. By the inductive hypothesis, we have $xm,ym\in M'$, and hence also $m\in M'$.

Finally, $M$ being Artinian implies that there are but finitely many nonzero summands in the direct sum above (if there were infinitely many, we could peel away summands one at a time to produce an infinite, strictly decreasing, chain of submodules of $M$), and the proof is complete.
\end{proof}

\begin{proposition}\label{prop2}  Let $R$ be a commutative ring, let $\mu$ be a sum-small module measure on $R$-$\mathrm{mod}$, and let $M$ be an Artinian minimally $\mu$-big $R$-module. Then $M$ is simple.
\end{proposition}

\begin{proof} By Lemma~\ref{lem2} and Lemma~\ref{min-big-lem}(1), $M=M[J]$ for some maximal ideal $J$ of $R$. For every positive integer $n$, set $M_n:=\{m\in M\colon J^nm=\{0\}\}$.

Suppose that $\mu(M_n)=0$ for every positive integer $n$. Because $\mu$ is sum-small, we see that $0=\mu(\sum_{n\in\mathbb{Z}^+}M_n)=\mu(M[J])=\mu(M)$, contradicting $M$ being minimally $\mu$-big. Thus for some positive integer $n$, we have $\mu(M_n)=1$. Minimality implies that $M=M_n$, and hence $J^n\subseteq \mathrm{Ann}_R(M)$. Since $\mathrm{Ann}_R(M)$ is a prime ideal of $R$, by Proposition~\ref{prop1}, and $J$ is maximal, we deduce that $J = \mathrm{Ann}_R(M)$. Thus $M$ is simple, by Proposition~\ref{prop1}. \end{proof}

We now give a short proof of the result of Anderson mentioned above.

\begin{theorem}[Anderson \cite{DA}]\label{thm2} Let $R$ be a commutative ring. Then every Artinian $R$-module is countably generated. \end{theorem}

\begin{proof} Define $\mu$ on $R$-mod by
\[
\mu(M):=
\begin{cases}
0 & \text{if } \ M \ \textrm{is countably generated,} \\
1 & \text{otherwise.} \\
\end{cases}
\]
It is immediate from Lemma \ref{lem1}(4) that $\mu$ is a module measure. Since a countable sum of countably generated modules is still countably generated, we see that $\mu$ is sum-small.

Now, suppose, seeking a contradiction, that there is an Artinian $R$-module $M$ which is not countably generated. Since $M$ is Artinian, we may assume without loss of generality that $M$ is minimally $\mu$-big. Then, by Proposition \ref{prop2}, $M$ is simple, which is absurd. \end{proof}

\begin{remark} In~\cite{DA} Anderson asserted that the above result is false, in general, for noncommutative rings. Then in~\cite{OAK2} Karamzadeh found necessary and sufficient conditions for an Artinian left module over a possibly noncommutative ring to be countably generated. \end{remark}

\section{Noetherian implies Artinian}

We used the concept of minimally $\mu$-big modules to prove the first two theorems of this paper, but have yet to give an application of maximally $\mu$-small modules. In this section we give such an application, and present a new proof that a commutative ring $R$ is Artinian if and only if $R$ is Noetherian of Krull dimension zero (the latter means that every prime ideal of $R$ is minimal). Our approach will use some basic facts about localization and minimal prime ideals. We refer the reader to \cite{RG} for details.

It is well-known that every Noetherian ring has but finitely many minimal prime ideals. In case the prime ideals are maximal, we give a new proof using ideal-theoretic techniques.

\begin{lemma}\label{lem6} A commutative zero-dimensional Noetherian ring has but finitely many prime ideals. \end{lemma}

\begin{proof} Suppose, seeking a contradiction, that there is a commutative zero-dimensional Noetherian ring $R$ with infinitely many prime ideals. Consider the sum $I:=\sum_{P \ \textrm{prime}}\mathrm{Ann}_R(P)$ of the annihilators of the prime ideals of $R$. Because $R$ is Noetherian, $I=\langle r_1,\ldots,r_n\rangle$ for some $r_1,\ldots,r_n\in R$. Now let $r=r_i$ for some $i$, $1\leq i\leq n$. Then $r=x_1+\cdots +x_k$, with each $x_i\in\mathrm{Ann}_R(P_i)$ for some prime ideal $P_i$. It follows that $r\in\mathrm{Ann}_R(P_1\cdots P_s)$ for some prime ideals $P_1,\ldots, P_s$ of $R$. Because $I$ is finitely generated, we deduce that there are prime ideals $P_1,\ldots, P_t$ of $R$ such that $I\subseteq\mathrm{Ann}_R(P_1\cdots P_t)$. By supposition, we can choose a prime ideal $Q$ distinct from the $P_i$. Then, as $R$ is zero-dimensional, $Q$ is a minimal prime ideal of $R$, and so, by a standard fact, there is a nonzero $r\in R$ such that $rQ=\{0\}$. However, $r\in I$, and hence $rP_1\cdots P_t=\{0\}$. But then $r(P_1\cdots P_t+Q)=rR=\{0\}$, which is absurd, as $R$ has an identity. \end{proof}

We require a final lemma, which gives a new proof of a more restrictive version of Theorem \ref{thm88} to follow.

\begin{lemma}\label{lem7} A commutative Noetherian ring with a unique prime ideal is Artinian. \end{lemma}

\begin{proof} Let $R$ be a commutative Noetherian ring with a unique prime ideal $P$. By Lemma \ref{lem1}(2), letting
\[
\mu(M):=
\begin{cases}
0 & \text{if } \ M \ \textrm{is Artinian,} \\
1 & \text{otherwise} \\
\end{cases}
\]
for each $R$-module $M$, defines a module measure on $R$-mod. Assume, seeking a contradiction, that $R$ is not Artinian. Because $R$ is Noetherian (and since the trivial ideal is Artinian as an $R$-module), there is an ideal $I$ of $R$ which is maximally $\mu$-small, that is, maximal with respect to being an Artinian $R$-module. Note that since $R$ is not Artinian, $I\neq R$. But also $I\neq P$, since otherwise both $P$ and $R/P$ would be Artinian $R$-modules (since $R/P$ is simple), contradicting $R$ not being Artinian. We deduce that $S:=R/I$ is a Noetherian ring with unique (nonzero) prime ideal $Q:=P/I$, such that every nonzero ideal of $S$ contains an infinite, strictly descending, chain of ideals of $S$. Because $Q$ is a finitely generated minimal prime ideal of $S$, there is a nonzero $s\in S$ such that $sQ=\{0\}$. Let $x\in\langle s\rangle$ be nonzero. Then $x=sy$ for some $y\in S$. We cannot have $y\in Q$, since then $x=sy=0$, and so, as $R$ is local, $y$ is a unit. This implies that $\langle x\rangle=\langle s\rangle$. But then $\langle s\rangle$ is a minimal (nonzero) ideal of $S$, contradicting the fact that every nonzero ideal of $S$ contains an infinite, strictly descending, chain of ideals of $S$, and the proof is complete. \end{proof}

We now give a new proof, as promised earlier, of the fact that a commutative ring is Artinian if and only if it is zero-dimensional and Noetherian. As we have already established the forward implication in Theorem~\ref{thm1}, we focus only on the reverse.

\begin{theorem}[Akizuki \cite{YA}]\label{thm88} Every commutative zero-dimensional Noetherian ring is Artinian. \end{theorem}

\begin{proof} Suppose that $R$ is a commutative zero-dimensional Noetherian ring, and let $\cdots I_3\subseteq I_2\subseteq I_1$ be a descending chain of ideals of $R$. For any prime ideal $P$ of $R$, the localization $R_P$ is a Noetherian ring with unique prime ideal $P_P$, and hence $R_P$ is also Artinian, by Lemma~\ref{lem7}. Now, by Lemma \ref{lem6}, $R$ has but finitely many prime ideals. It follows that there is a positive integer $k$ such that for every prime ideal $P$ of $R$ and every integer $n\geq k$, we have $(I_n)_P=(I_k)_P$. Since this holds for every prime ideal $P$ of $R$, we deduce that $I_n=I_k$, and hence $R$ is Artinian. \end{proof}

\section{Cardinalities of Artinian modules}

We conclude the paper with an application of our results above, where we study possible cardinalities of an Artinian module $M$ relative to the operator ring $R$. Recall from Lemma~\ref{lem2} that every Artinian $R$-module $M$ decomposes as $M=M[J_1]\oplus\cdots\oplus M[J_n]$, for some maximal ideals $J_1,\ldots,J_n$ of $R$. Thus we may consider only the components $M[J_i]$, without loss of generality. We can, moreover, restrict to modules over local rings, as each $M[J_i]$ is naturally a module over the localization $R_{J_i}$.

\begin{theorem}\label{thm3} Let $R$ be a commutative local ring, with unique maximal ideal $J$, and let $M$ be a nontrivial Artinian $R$-module. Then the following hold.
\begin{enumerate}
\item If $R/J$ is finite, then $M$ is countable.
\item If $R/J$ is infinite, then $|M|=|R/J|$.
\end{enumerate}
\end{theorem}

\begin{proof} Since $R$ is local, Lemma~\ref{lem2} yields that $M=M[J]$. As in the proof of Proposition~\ref{prop2}, for every positive integer $n$, let $M_n:=\{m\in M\colon J^nm=\{0\}\}$, and note that $M=\bigcup_{n\in\mathbb{Z}^+}M_n$. Now, $M_1$ is annihilated by $J$, and so is naturally a (Artinian) vector space over the field $R/J$. Since $M_1$ is Artinian, it is finite-dimensional over $R/J$, and thus $|M_1|=|R/J|^{k_1}$ for some positive integer $k_1$. Similarly, $M_2/M_1$ is annihilated by $J$ and is also a finite-dimensional vector space over $R/J$. Thus $|M_2/M_1|=|R/J|^{k}$ for some positive integer $k$. Hence
\[|M_2|=|R/J|^k\cdot|R/J|^{k_1}=|R/J|^{k+k_1}=|R/J|^{k_2},\]
where $k_2:=k+k_1$. It follows easily, by induction, that for every positive integer $n$, there is a positive integer $k_n$ such that $|M_n|=|R/J|^{k_n}$.

If $R/J$ is finite, then so is each $M_n$, and since $M=\bigcup_{n\in\mathbb{Z}^+}M_n$, we see that $M$ is countable. Now suppose that $R/J$ is infinite. Then basic cardinal arithmetic implies that $|M_n|=|R/J|^{k_n}=|R/J|$ for every positive integer $n$. Thus
\[|M|=\Big|\bigcup_{n\in\mathbb{Z}^+}M_n\Big|=\aleph_0\cdot|R/J|=|R/J|,\]
and this completes the proof. \end{proof}

\bibliographystyle{plain}

\end{document}